\documentclass{amsart}
\usepackage{csquotes}
\newtheorem{theorem}{Theorem}[section]
\newtheorem{lemma}[theorem]{Lemma}

\theoremstyle{definition}
\newtheorem{definition}[theorem]{Definition}

\theoremstyle{proposition}
\newtheorem{proposition}[theorem]{Proposition}
\theoremstyle{remark}
\newtheorem{remark}[theorem]{Remark}

\theoremstyle{corollary}
\newtheorem{corollary}[theorem]{Corollary}

\numberwithin{equation}{section}

\begin{document}
	
	\title[Approximation with Multivariate Fractal Functions]{Multivariate fractal interpolation functions: Some approximation aspects 
		and an associated fractal interpolation operator}

	
	
	
	
	\author{K. K. Pandey}
	\address{Department of Mathematics, IIT Delhi, New Delhi, India 110016}
	
	\email{kshitij.sxcr@gmail.com}
	
	
	\author{P. Viswanathan}
	\address{Department of Mathematics, IIT Delhi, New Delhi, India 110016}
	
	\email{viswa@maths.iitd.ac.in}
	
	
	
	
	
	\begin{abstract}
		The natural kinship between  classical theories of interpolation and approximation  is well explored. In contrast to this,   the interrelation between interpolation and approximation is subtle and this duality is relatively obscure in the context of fractal interpolation. The notion of $\alpha$-fractal function provides a proper foundation for the approximation theoretic facet of univariate fractal interpolation functions (FIFs).  However, no comparable approximation theoretic aspects of FIFs has been developed for functions of several variables. The current article intends to open the door for intriguing interaction between approximation theory and multivariate FIFs. To this end, in the first part of this article we develop a general framework to construct multivariate FIF, which is amenable to  provide a multivariate analogue of the $\alpha$-fractal function. Multivariate $\alpha$-fractal functions provide  a parameterized family of fractal approximants  associated to  a given multivariate continuous function. Some elementary aspects of the  multivariate  fractal nonlinear (not necessarily linear) interpolation operator that sends a continuous function defined on a hyper-rectangle to its fractal analogue is studied. 
	\end{abstract}
	
	\maketitle

	
	.
	
	\section{INTRODUCTION}
	The first systematic study on interpolation of univariate data with continuous functions whose graphs are self-referential sets (fractals) -  popularly known as fractal interpolation - has its origin in Barnsley's fundamental paper on fractal functions and interpolation \cite{MF1}. During the past three decades, various questions concerning the fractal interpolation have led to numerous generalizations to the original setting by Barnsley; for a lucid exposition we refer the reader to the book \cite{PMbook}.   One of the features of fractal interpolation that distinguishes it from various standard interpolation techniques is its ability to produce both smooth and nonsmooth interpolants. It is worth noting that there are only a very few methods that produce nonsmooth interpolating functions, another popular method being subdivision schemes \cite{DL}. Recently, attempts have been made to link subdivision schemes and fractal interpolation \cite{DLM, LDV}. The Hausdorff dimension of the graph of a fractal interpolant provides an additional index to measure the complexity of a signal, for instance, EEG signal \cite{DMIV}. Further, the smooth fractal interpolation supplements and subsumes the theory of splines and Hermite interpolation \cite{Chand2,PM,NS}.

	\par Classical theories of interpolation and approximation often appear as two sides of a same coin, the result about the one frequently imply the results about the other.  Indeed, at a basic level both are, in essence, one and the same.  This duality between interpolation and approximation seems to be more subtle in the fractal setting. Fruitful interactions between the notion of univariate FIF and classical approximation theory took place via a suitable subclass of fractal interpolation function. This subclass was brought out by Barnsley himself; see, for instance, \cite{MF1}, especially his remarks in Example 2 on p. 309 there. The aforementioned  subclass of FIFs  was later named as $\alpha$-fractal functions to reflect the vectorial parameter $\alpha$ that influences the Hausdorff dimension of the graph of a FIF. Substantial extensions of this theme have been carried out by Navascue\'{e}s and her coworkers \cite{M1,M2,M3,PV2}. It is our opinion that the concept of $\alpha$-fractal functions assisted fractal interpolation to find applications in other fields of mathematics, that belong, in broad sense, to the topic of approximation of functions in various function classes, for instance, in the theory of bases and frames \cite{N1,N2}.  Another interesting theoretical and practical ramification of these findings is the fact that the fractal functions can be used for constrained approximation \cite{VCN}. Recent years have witnessed a renewed level of interest in fractal interpolation, in particular, in the study of $\alpha$-fractal functions.

	\par 
	Parallel with, or perhaps even prior to,  the investigations on approximation theoretic aspects of univariate FIF through the notion of $\alpha$-fractal functions,  attempts have been made to study  multivariate analogues of the fractal interpolation, especially the bivariate FIFs or fractal surfaces. The study of multivariate FIFs, even the bivariate case, is more difficult and approaches are less obvious; see, for instance, \cite{PB,Chand,Dalla,Feng,Mal,PM,Metz,Xie,Zhao}.
	In most cases, the construction is confined to the case wherein some suitable restriction on interpolation points is imposed or maps in the IFS use equal scaling factors. 
	In \cite{Ruan} the author gives a more general framework to construct the bivariate fractal interpolation functions for data on the rectangular grids. Our interest in \cite{Ruan} is attributed to the fact that the formalism therein can be easily adapted to obtain the bivariate $\alpha$-fractal function, an interlude to the study of approximation theoretic aspects of bivariate fractal interpolation; see the recent works reported in \cite{Verma1,Verma2}.

	\par In contrast to the univariate and bivariate theory of FIFs, higher dimensional analogues are scarce in the literature. Hardin and Massopust \cite{HM,PMbook} constructed fractal interpolation functions from a polygonal $D \subset \mathbb{R}^n$ to $\mathbb{R}^m$ using suitable triangulations of $D$. Construction of FIFs on arbitrary interpolation points placed on rectangular grids of $\mathbb{R}^n$ is undertaken in \cite{PB1}. Both these constructions are based on the concept of recurrent IFSs \cite{MFBook, PMbook} and ensuring continuity of such multivariate fractal functions presents geometric complications beyond those which arise for similar univariate fractal functions or univariate vector-valued fractal functions.

	\par Principle aim of this article is to initiate an interaction between multivariate FIFs and multivariate approximation theory and thereby expose some interesting approximation theoretic considerations of multivariate FIFs. Having gained some experience with fractal approximation theory of univariate and bivariate functions, one could easily anticipate that the development of multivariate analogue of $\alpha$-fractal function could be the first and foremost step to accomplish this. However, the impediment is that a general framework to construct  multivariate FIFs that is appropriate to provide  the notion of multivariate $\alpha$-fractal function is unavailable. The  constructions of multivariate FIFs hinted at the previous paragraph  do not seem to lend themselves  the $\alpha$-fractal function formalism of the multivariate FIFs. In the first part of this article we overcome the aforementioned obstacle by developing a general  framework for the construction of  multivariate FIFs. The ``computational cost" for ensuring continuity of the interpolant in this construction is  similar to that in the construction of the univariate FIFs proposed by Barnsley. The corresponding problem for bivariate case was treated in \cite{Ruan} to which work the first part of the article  may be considered a sequel.  However, our interest is in the multivariate analogue of $\alpha$-fractal function because  it provides a vehicle to interact with approximation theory. Our theory has been designed to establish a rigorous definition of the multivariate $\alpha$-fractal function.

	\section{Multivariate fractal interpolation functions}
	Let $n \ge 2$ be a natural number. Consider a data set  $$\Delta = \big\{(x_{1,i_1},x_{2,i_2},\ldots,x_{n,i_n},y_{i_{1}i_{2}\ldots i_{n}}): i_{k}=0,1,\dots,N_{k}; k=1,\dots,n\big\}$$  such that $$a_k=x_{k,0}< x_{k,1} < \dots <x_{k,N_{k}}=b_k$$ for each $k = 1,2,\ldots,n;~n\geq 2.$ For $k = 1,2,\ldots n$, set $I_{k}=[a_k,b_k]$. 
	To simplify the notation, for $m\in\mathbb{N},$ we write
	\begin{equation*}
		\begin{split}
			& \Sigma_{m}=\{1,2,\dots,m\},~~ \Sigma_{m,0}=~\{0,1,\dots m \},\\ & \partial \Sigma_{m,0}=~\{0,m\},~~ \text{int} \Sigma_{m,0}=~\{1,2,\dots,m-1\}.
		\end{split}
	\end{equation*} 
	We shall denote by $I_{k,i_k}$, the typical subinterval of $I_k$ determined by the partition $\{x_{k,0}, x_{k,1}, \dots, x_{k,N_{k}} \}$,
	$I_{k,i_k} =[x_{k,i_k-1},x_{k,i_k}]$ for $i_k \in \Sigma_{N_{k}}.$ For any $i_k \in \Sigma_{N_k},$ let $u_{k,i_k}:I_{k} \rightarrow I_{k,i_k}$ be an affine map satisfying 
	
	\begin{equation}\label{constraintu1}
		\begin{cases}
			\begin{aligned}
				u_{k,i_k}(x_{k,0})=x_{k,i_k -1}~\text{and}~  u_{k,i_k}(x_{k,N_k})=x_{k,i_k}, ~~ \text{if}~ i_k~ \text{is~ odd},\\
				u_{k,i_k}(x_{k,0})=x_{k,i_k}~\text{and}~ u_{k,i_k}(x_{k,N_k})=x_{k,i_k -1}, ~~ \text{if}~ i_k~ \text{is~ even},
			\end{aligned} 
		\end{cases}
	\end{equation}
	
	
	\begin{equation}\label{constraintu3}
		\begin{aligned}
			\big|u_{k,i_k}(x)-u_{k,i_k}(x')\big| \le \alpha_{k,i_k}|x -x'|, ~~~~ \forall ~~x, x' \in I_k,
		\end{aligned} 
	\end{equation}
	where $0 \leq \alpha_{k,i_k }< 1$ is a constant. Using the definition of $u_{k,i_k}$, it is easy to check that
	\begin{equation}\label{constraintu4}
		\begin{aligned}
			u_{k,i_k}^{-1}(x_{k,i_k})=u_{k,i_k+1}^{-1}(x_{k,i_k}),~~ \forall ~~i_k \in \text{int}\Sigma_{N_k,0}.
		\end{aligned}                
	\end{equation}
	Let $\tau: \mathbb{Z}\times \{0,N_1,N_2,\ldots,N_n\} \rightarrow \mathbb{Z}$ be defined by 
	\begin{equation}
		\begin{cases}
			\begin{aligned}
				&\tau(i,0)= i-1~~ \text{and} ~~\tau(i,N_k)= i, ~~ \text{if}~~ i~~ \text{is odd},\\ 
				&\tau(i,0)= i,~~\text{and}~~ \tau(i,N_k)= i-1, ~~\text{ if}~~ i~~ \text{is even}.
			\end{aligned} 
		\end{cases}
	\end{equation}
	Using the above notation we see that $$ u_{k,i_k}(x_{k,j_k}) = x_{k,\tau(i_k,j_k)},~~ \forall~~ i_k \in \Sigma_{N_{k}}, j_k \in \partial\Sigma_{N_{k},0},~~k \in \Sigma_{n}.$$               
	Let $K := \Big(\prod_{k=1}^{n}I_k\Big)\times\mathbb{R}.$ For each $(i_1,i_2,\ldots, i_n) \in \prod_{k=1}^{n}\Sigma_{N_{k}},$ let $F_{i_1 i_2 \ldots i_n}:K \rightarrow \mathbb{R}$ be a continuous function satisfying the following conditions. 
	\begin{equation}\label{constraintF}
		\begin{aligned}
			F_{i_1 i_2 \ldots i_n}(x_{1,j_1},x_{2,j_2},\ldots, x_{n,j_n},y_{j_1 j_2 \ldots j_n}) = y_{\tau(i_1,j_1)\tau(i_2,j_2)\ldots\tau(i_n,j_n)}, 
		\end{aligned}
	\end{equation}
	$\text{for all}~~ (j_1, j_2, \ldots, j_n) \in \prod_{k=1}^{n}\partial\Sigma_{N_{k},0}$ and 
	\begin{equation}\label{constraintF1}
		\begin{aligned}
			\big |F_{i_1 i_2 \ldots i_n}(x_1,x_2,\ldots, x_n,y) - F_{i_1 i_2 \ldots i_n}(x_1,x_2,\ldots,x_n,y')\big| &\le \gamma_{i_1 i_2 \ldots i_n} |y - y'|,\\ 
		\end{aligned}
	\end{equation}
	$\text{for all}~~ (x_1,x_2,\ldots,x_n) \in \prod_{k=1}^{n}I_k~~ \text{and}~~ y,y' \in \mathbb{R},$ where $0 \leq \gamma_{i_1 i_2 \ldots i_n}< 1$ is a constant.\\
	Finally, for each $(i_1,i_2,\ldots, i_n) \in \prod_{k=1}^{n}\Sigma_{N_{k}},$ we define $W_{i_1 i_2 \ldots i_n}:K \rightarrow K$ by 
	\begin{equation}\label{IFS}
		\begin{aligned}
			W_{i_1 i_2 \ldots i_n}(x_1,x_2,\ldots,x_n,y)= &~\big(u_{1,i_1}(x_1),u_{2,i_2}(x_2),\ldots,u_{n,i_n}(x_n),\\& F_{i_1 i_2 \ldots i_n}(x_1,x_2,\ldots,x_n,y)\big).
		\end{aligned}
	\end{equation}
	and consider the Iterated Function System (IFS) $$\Big\{K, W_{i_1 i_2 \ldots i_n} : (i_1, i_2, \ldots i_n) \in \prod_{k=1}^{n}\Sigma_{N_{k}} \Big\}.$$ 
	For the definition of IFS and its role in the theory of univariate fractal interpolation function, the interested reader may consult \cite{MF1}. The following theorem is a multivariate analogue of the construction of univariate FIFs originally appeared in \cite{MF1} and its bivariate extension studied in \cite{Ruan}. We provide an expanded rendition of the arguments for the sake of expositions and record. 
	\begin{theorem}\label{FIF}
		Let $\Delta = \big\{(x_{1,i_1},x_{2,i_2},\ldots,x_{n,i_n},y_{i_{1}i_{2}\ldots i_{n}}): i_{k}=0,1,\dots,N_{k}; k=1,\dots,n\big\}$ be a prescribed multivariate data set and $\{K,  W_{i_1 i_2 \ldots i_n} : (i_1, i_2, \ldots, i_n) \in \prod_{k= 1}^{n}\Sigma_{N_{k}} \}$ be the IFS associated to it, as  defined above. Assume that for each $(i_1,i_2,\ldots, i_n) \in \prod_{k=1}^{n}\Sigma_{N_{k}} $, the map  $F_{i_1 i_2 \ldots i_n}$ satisfy the following matching conditions:\\
		For all $ i_k \in \text{int}\Sigma_{N_k,0}$, $ 1\leq k\leq n$, $(i_1,i_2,\ldots, i_n) \in \prod_{k=1}^{n}\Sigma_{N_{k}}$ and $x_k^*=u_{k,i_k}^{-1}(x_{k,i_k})=u_{k,i_k+1}^{-1}(x_{k,i_k}),$
		\begin{equation}
			\begin{aligned}\label{matching}
				&~ F_{i_1\ldots i_k \ldots i_n}(x_1,\ldots,x_{k-1}, x_k^*, x_{k+1},\ldots x_n,y) \\ &= F_{i_1\ldots i_k+1 \ldots i_n}(x_1,\ldots, x_{k-1}, x_k^*, x_{k+1}\ldots x_n,y),
			\end{aligned}  
		\end{equation}
		where $ (x_1,\ldots x_{k-1},x_{k+1}\ldots, x_n) \in \prod_{j=1,j\neq k}^{n}I_j$ and $ y \in \mathbb{R}.$
		Then there exists a unique continuous function $\tilde{f}: \prod_{k=1}^{n}I_k\rightarrow \mathbb{R}  $ such that
		\begin{enumerate}
			\item $\tilde{f}$ interpolates the given multivariate data. That is,  $$\tilde{f}(x_{1,i_1},x_{2,i_2},\ldots,x_{n,i_n})= y_{i_1i_2\ldots i_n},~~\forall~~(i_1,i_2,\ldots, i_n) \in \prod_{k = 1}^{n}\Sigma_{N_k,0} .$$  
			\item The graph of $\tilde{f}$ defined by $$G=\Big\{\big(x_1,x_2,\ldots,x_n, \tilde{f}(x_1,x_2,\ldots,x_n)\big):(x_1,x_2,\ldots,x_n) \in \prod_{k=1}^{n}I_k \Big\}$$ is self-referential in the following sense: $G$ is the union of transformed copies of itself given by
			$$G= \bigcup_ {(i_1,i_2,\ldots, i_n) \in \prod_{k=1}^{n}\Sigma_{N_k}} W_{i_1i_2\ldots i_n}(G).$$
		\end{enumerate}
	\end{theorem}
	
	\begin{proof}
		For convenience of the reader, we shall divide the proof into several steps. \\
		Step I: (\textbf{Considering Appropriate function space and Read-Bajraktarevi\'{c} operator}):	\\
		Let $\mathcal{C}\big(\prod_{k = 1}^{n}I_k\big)$ denote the Banach space of all real-valued continuous functions defined on the $n$-dimensional  hyperrectangle $\prod_{k = 1}^{n}I_k$
		endowed with the uniform norm. The subset 
		\begin{equation*}
			\begin{split}
				\mathcal{C}^*\big(\prod_{k = 1}^{n}I_k\big) = &~ \Big\{g \in \mathcal{C}\big(\prod_{k = 1}^{n}I_k\big): g(x_{1,j_1},x_{2,j_2},\ldots,x_{n,j_n}) = y_{j_{1}j_{2}\ldots j_{n}}:\\& (j_1,j_2,\ldots,j_n)\in\prod_{k=1}^{n}\partial\Sigma_{N_k,0}\Big\}.
			\end{split}
		\end{equation*}
		is a closed (and hence complete) metric subspace of $ 
		\mathcal{C}\big(\prod_{k = 1}^{n}I_k\big)$ with the uniform metric.
		Consider  the so-called Read-Bajraktarevi$\acute{c}$ (RB) operator (see also \cite{MF1}) $$T: \mathcal{C}^*(\prod_{k =1}^{n}I_k) \rightarrow \mathcal{C}^*(\prod_{k=1}^{n}I_k)$$ defined by
		\begin{equation}\label{RB}
			\begin{aligned}
				(Tg)(x_1,x_2,\ldots,x_n)=F_{i_1 i_2 \ldots i_n}\Big( u_{1,i_1}^{-1}(x_1),&\ldots,u_{n,i_n}^{-1}(x_n),g\big(u_{1,i_1}^{-1}(x_1),\ldots,u_{n,i_n}^{-1}(x_n)\big)\Big),
			\end{aligned}
		\end{equation}
		for all $(x_1,x_2,\ldots,x_n) \in \prod_{k=1}^{n}I_{k,i_k}$  and 
		$ (i_1,\ldots, i_n) \in \prod_{k=1}^{n}\Sigma_{N_{k}}.$\\
		\noindent
		Step II (\textbf{Proving that $T$ is well-defined}):	\\ Let $g\in \mathcal{C}^*(\prod_{k=1}^nI_k)$. For simplicity, let us consider $X = (x_1,\ldots,x_r, \dots,x_s, \dots, x_n)\in \prod_{k=1}^{n}I_{k,i_{k}}$ such that $x_{r}\in I_{r,i_r}\cap I_{r,i_r+1}$ and $x_{s}\in I_{s,i_s}\cap I_{s,i_s+1}$ for some $r,s \in \Sigma_n$ and $(i_r,i_s) \in\text{int}\Sigma_{N_{r},0}\times\text{int}\Sigma_{N_{s},0}.$ This is possible if and only if $x_r = x_{r,i_r}$ and $x_s = x_{s,i_s}.$ Without loss of generality we assume that $r<s$ and deal with the following four possible cases.\\
		
		Case 1. Treating $x_r = x_{r,i_r}$ as a point in $ I_{r,i_r}$ and $x_s = x_{s,i_s}$ as a point in $I_{s,i_s}$, we have
		\begin{equation*}
			\begin{aligned}
				T(g)(X) = F_{ i_1\ldots i_r \ldots i_s \ldots i_n}\Big(&  u_{1,i_1}^{-1}(x_1),\ldots,u_{r,i_r}^{-1}(x_{r,i_r}),\ldots,u_{s,i_s}^{-1}(x_{s,i_s})\ldots,u_{n,i_n}^{-1}(x_n),\\
				&g\big(u_{1,i_1}^{-1}(x_1),\ldots,u_{n,i_n}^{-1}(x_n)\big)\Big).
			\end{aligned}
		\end{equation*}
		Case 2. Consider $x_r = x_{r,i_r}$ as a point in $ I_{r,i_r+1}$ and $x_s = x_{s,i_s}$ as a point in  $I_{s,i_s}.$ Bearing (\ref{constraintu4}) and (\ref{matching}) in mind,  one gets
		\begin{equation*}
			\begin{aligned}
				T(g)(X) &= F_{i_1\ldots i_r+1\ldots i_s \ldots i_n}\Big( u_{1,i_1}^{-1}(x_1),\ldots u_{r,i_r+1}^{-1}(x_{r,i_r}),\ldots,u_{s,i_s}^{-1}(x_{s,i_s})\ldots,u_{n,i_n}^{-1}(x_n),\\
				&g\big(u_{1,i_1}^{-1}(x_1),\ldots,u_{n,i_n}^{-1}(x_n)\big)\Big)\\
				&= F_{i_1\ldots i_r \ldots i_s \ldots i_n}\Big( u_{1,i_1}^{-1}(x_1),\ldots u_{r,i_r}^{-1}(x_{r,i_r}),\ldots,u_{s,i_s}^{-1}(x_{s,i_s})\ldots,u_{n,i_n}^{-1}(x_n),\\
				&g\big(u_{1,i_1}^{-1}(x_1),\ldots,u_{n,i_n}^{-1}(x_n)\big)\Big).
			\end{aligned}
		\end{equation*}
		Case 3. Considering $x_r = x_{r,i_r}$ as an element in $ I_{r,i_r}$ and $x_s = x_{s,i_s}$ as an element  in $I_{s,i_s+1}$, similar to the previous case we have
		\begin{equation*}
			\begin{aligned}
				T(g)(X) &= F_{i_1\ldots i_r\ldots i_s+1 \ldots i_n}\Big( u_{1,i_1}^{-1}(x_1),\ldots u_{r,i_r}^{-1}(x_{r,i_r}),\ldots,u_{s,i_s+1}^{-1}(x_{s,i_s})\ldots,u_{n,i_n}^{-1}(x_n),\\
				&g\big(u_{1,i_1}^{-1}(x_1),\ldots,u_{n,i_n}^{-1}(x_n)\big)\Big)\\
				&= F_{i_1\ldots i_r \ldots i_s \ldots i_n}\Big( u_{1,i_1}^{-1}(x_1),\ldots u_{r,i_r}^{-1}(x_{r,i_r}),\ldots,u_{s,i_s}^{-1}(x_{s,i_s})\ldots,u_{n,i_n}^{-1}(x_n),\\
				&g\big(u_{1,i_1}^{-1}(x_1),\ldots,u_{n,i_n}^{-1}(x_n)\big)\Big).
			\end{aligned}
		\end{equation*}
		Case 4. Finally, let us view $x_r = x_{r,i_r}$ as a point in $I_{r,i_r+1}$ and $x_s = x_{s,i_s}$ as  a point in $I_{s,i_s+1}$. Using (\ref{constraintu4}) and (\ref{matching})  we obtain
		\begin{equation*}
			\begin{aligned}
				T(g)(X) &= F_{i_1\ldots i_r+1\ldots i_s+1 \ldots i_n}\Big( u_{1,i_1}^{-1}(x_1),\ldots u_{r,i_r+1}^{-1}(x_{r,i_r}),\ldots,u_{s,i_s+1}^{-1}(x_{s,i_s})\ldots,u_{n,i_n}^{-1}(x_n),\\
				&g\big(u_{1,i_1}^{-1}(x_1),\ldots,u_{n,i_n}^{-1}(x_n)\big)\Big)\\
				&= F_{i_1\ldots i_r \ldots i_s+1 \ldots i_n}\Big( u_{1,i_1}^{-1}(x_1),\ldots u_{r,i_r}^{-1}(x_{r,i_r}),\ldots,u_{s,i_s+1}^{-1}(x_{s,i_s+1})\ldots,u_{n,i_n}^{-1}(x_n),\\
				&g\big(u_{1,i_1}^{-1}(x_1),\ldots,u_{n,i_n}^{-1}(x_n)\big)\Big)\\
				&= F_{i_1\ldots i_r \ldots i_s \ldots i_n}\Big( u_{1,i_1}^{-1}(x_1),\ldots u_{r,i_r}^{-1}(x_{r,i_r}),\ldots,u_{s,i_s}^{-1}(x_{s,i_s})\ldots,u_{n,i_n}^{-1}(x_n),\\
				&g\big(u_{1,i_1}^{-1}(x_1),\ldots,u_{n,i_n}^{-1}(x_n)\big)\Big).
			\end{aligned}
		\end{equation*}
		We see that the value of $T(g)(X)$ is determined univocally in all the four cases. Similarly, all other possibilities can be worked out to conclude that $T(g)$ is well defined on the boundary of $\prod_{k = 1}^{n}I_{k,i_k}$. Furthermore, $T(g)$ is continuous on $\prod_{k=1}^{n} I_{k}$.\\
		Let $(i_1,i_2,\ldots,i_n)\in \prod_{k =1}^{n}\Sigma_{N_{k},0}$.  Choose $(j_1,j_2,\ldots,j_n)\in \prod_{k =1}^{n}\partial\Sigma_{N_{k},0}$ such that $(i_1,i_2,\ldots,i_n) = (\tau(i_1,j_1),\tau(i_2,j_2),\ldots,\tau(i_n,j_n))$. By the definition  of $\tau$, we have
		$$(x_{1,j_1},x_{2,j_2},\ldots,x_{n,j_n}) = \big(u_{1,i_1}^{-1}(x_{1,i_1}),u_{2,i_2}^{-1}(x_{2,i_2}),\ldots,u_{n,i_n}^{-1}(x_{n,i_n})\big).$$
		Thus,
		\begin{equation*}
			\begin{aligned}
				T(g)(x_{1,i_1},\ldots,x_{n,i_n})=&~ F_{i_1 i_2 \ldots i_n}\Big( u_{1,i_1}^{-1}(x_{1,i_1}),u_{2,i_2}^{-1}(x_{2,i_2}),\ldots,u_{n,i_n}^{-1}(x_{n,i_n}),\\
				&g\big(u_{1,i_1}^{-1}(x_{1,i_1}),u_{2,i_2}^{-1}(x_{2,i_2}),\ldots,u_{n,i_n}^{-1}(x_{n,i_n})\big)\Big)\\
				=&~ F_{i_1 i_2 \ldots i_n}\Big(x_{1,j_1},x_{2,j_2},\ldots,x_{n,j_n},g\big(x_{1,j_1},x_{2,j_2},\ldots,x_{n,j_n}\big)\Big)\\
				=&~ F_{i_1 i_2 \ldots i_n}\Big(x_{1,j_1},x_{2,j_2},\ldots,x_{n,j_n},y_{j_1j_2\ldots j_n}\Big)\\
				=&~ y_{\tau(i_1,j_1)\tau(i_2,j_2)\ldots\tau(i_n,j_n)}\\
				=&~ y_{i_1i_2\ldots i_n},
			\end{aligned}
		\end{equation*}
		showing that $T(g)$ interpolates the data in $\Delta$ for all $g\in \mathcal{C}^*(\prod_{k=1}^{n}I_k).$ In particular, $T$ maps $\mathcal{C}^*(\prod_{k=1}^{n}I_k)$ into $\mathcal{C}^*(\prod_{k=1}^{n}I_k).$ \\
		Step III (\textbf{Proving that $T$ is a contraction}):\\ Let $g,h \in \mathcal{C}^*(\prod_{k=1}^{n}I_k)$ and $X = (x_{1,i_1},\ldots,x_{n,i_n}) \in \prod_{k=1}^{n}I_{k,i_k}.$ Using  (\ref{constraintF1}) and (\ref{RB}) 
		we have
		\begin{equation*}
			\begin{aligned}
				\big|T(g) (X)-T(h)(X)\big| &= \Big|F_{i_1 i_2 \ldots i_n}\Big( u_{1,i_1}^{-1}(x_1),\ldots,u_{n,i_n}^{-1}(x_n),g\big(u_{1,i_1}^{-1}(x_1),\ldots,u_{n,i_n}^{-1}(x_n)\big)\Big)-\\
				&F_{i_1 i_2 \ldots i_n}\Big( u_{1,i_1}^{-1}(x_1),\ldots,u_{n,i_n}^{-1}(x_n),h\big(u_{1,i_1}^{-1}(x_1),\ldots,u_{n,i_n}^{-1}(x_n)\big)\Big)\Big|\\
				&\leq \alpha_{i_1 i_2 \ldots i_n}\Big|g\big(u_{1,i_1}^{-1}(x_1),\ldots,u_{n,i_n}^{-1}(x_n)\big)-h\big(u_{1,i_1}^{-1}(x_1),\ldots,u_{n,i_n}^{-1}(x_n)\big)\Big|\\
				&\leq \|\alpha\|_{\infty}\|g-h\|_{\infty},
			\end{aligned}
		\end{equation*}
		where $\|\alpha\|_{\infty} = \max\{\alpha_{i_1 i_2 \ldots i_n}: (i_1, i_2, \ldots, i_n) \in \prod_{k=1}^{n}\Sigma_{N_k}\}.$ Since the above inequality holds for all $X \in \prod_{k=1}^{n}I_k$, it follows that
		\begin{equation*}
			\|T(g)-T(h)\|_{\infty} \leq \|\alpha\|_{\infty}\|g-h\|_{\infty}.
		\end{equation*}
		This yields that $T$ is a contraction on $\mathcal{C}^*(\prod_{k=1}^{n}I_k).$\\
		Step IV (\textbf{Proving the existence by the application of the Banach fixed point theorem})\\
		By the Banach fixed point theorem $T$ has a unique fixed point. That is,  there exists a unique function $\tilde{f}\in \mathcal{C}^*(\prod_{k=1}^{n}I_k)$ such that
		\begin{equation}\label{selfreferential}
			\begin{aligned}
				\tilde{f}(x_1,\ldots,x_n)=&F_{i_1 i_2 \ldots i_n}\Big( u_{1,i_1}^{-1}(x_1),\ldots,u_{n,i_n}^{-1}(x_n),\tilde{f}\big(u_{1,i_1}^{-1}(x_1),\ldots,u_{n,i_n}^{-1}(x_n)\big)\Big),\\
				&\forall ~~ (x_1,\ldots x_n) \in \prod_{k=1}^{n}I_{k,i_k} \hspace{0.2cm} \text{and}\hspace{0.2cm} (i_1,\ldots, i_n) \in \prod_{k=1}^{n}\Sigma_{N_{k}}.
			\end{aligned}
		\end{equation}
		Writing $X = (x_1,x_2,\ldots,x_n)$, $ u_{i_1,\ldots i_n}^{-1}(X) = \big(u_{1,i_1}^{-1}(x_1),\ldots,u_{n,i_n}^{-1}(x_n)\big)$ and $u_{i_1,\ldots i_n}(X) = \big(u_{1,i_1}(x_1),\ldots,u_{n,i_n}(x_n)\big),$ we have the self-referential equation
		\begin{equation}\label{selfreferential1}
			\begin{aligned}
				\tilde{f}(X) = F_{i_1 i_2 \ldots i_n}\big(u_{i_1,\ldots i_n}^{-1}(X),\tilde{f}(u_{i_1,\ldots i_n}^{-1}(X))\big),~~ \forall ~~X \in \prod_{k=1}^{n}I_{k,i_k} \hspace{0.2cm} \text{and} \hspace{0.2cm} (i_1,\ldots, i_n) \in \prod_{k=1}^{n}\Sigma_{N_{k}}.
			\end{aligned}
		\end{equation}
		Equivalently,
		\begin{equation}\label{selfreferential2}
			\begin{aligned}
				\tilde{f}\big(u_{i_1,\ldots i_n}(X)\big) = F_{i_1 i_2 \ldots i_n}\big(X,\tilde{f}(X)\big), \quad \forall ~~X \in \prod_{k=1}^{n}I_k \hspace{0.2cm} \text{and} \hspace{0.2cm} (i_1,\ldots, i_n) \in \prod_{k=1}^{n}\Sigma_{N_{k}}.
			\end{aligned}
		\end{equation}
		Let $G=\big\{(X, \tilde{f}(X)): X \in \prod_{k =1}^{n}I_k \big\}$ be the graph of $\tilde{f}.$ In view of (\ref{IFS}) and (\ref{selfreferential2}) 
		\begin{equation*}
			\begin{aligned}
				&\bigcup{\Big\{W_{i_1i_2\ldots i_n}(G) : {(i_1,i_2,\ldots,i_n) \in \prod_{k=1}^{n}\Sigma_{N_k}}\Big\}}\\
				&= \bigcup{ \Big\{W_{i_1i_2\ldots i_n}(X,\tilde{f}(X)) : X \in \prod_{k =1}^{n}I_k, {(i_1,i_2,\ldots,i_n)\in\prod_{k=1}^{n}\Sigma_{N_k}}\Big\}}\\
				&= \bigcup{\Big\{\big(u_{i_1\ldots i_n}(X),F_{i_1 i_2 \ldots i_n}(X,\tilde{f}(X))\big): X \in \prod_{k =1}^{n}I_k,{(i_1,i_2,\ldots,i_n)\in\prod_{k=1}^{n}\Sigma_{N_k}}\Big\}}\\
				&= \bigcup\big{\{\big(u_{i_1\ldots i_n}(X),\tilde{f}(u_{i_1,\ldots i_n}(X))\big): X \in \prod_{k =1}^{n}I_k,{(i_1,i_2,\ldots,i_n) \in \prod_{k=1}^{n}\Sigma_{N_k}}\Big\}}\\
				&= \bigcup \Big\{(X,f(X)): X \in \prod_{k=1}^n I_k\Big\}\\
				&= G,
			\end{aligned}
		\end{equation*}completing the proof.
	\end{proof}
	\begin{definition}
		The function  $\tilde{f}$ appearing in the previous theorem is termed a multivariate FIF.
	\end{definition}
	
	\begin{remark}
		As mentioned in the introductory section, in contrast to the univariate and bivariate settings,   a very few  studies addressed the multivariate fractal interpolation; \cite{PB1} and \cite{HM} are worth mentioning. In \cite{PB1} the authors construct multivariate fractal interpolation functions on the hyperrectangle $[0,1]^n$.  However, the approach in these references is  primarily based on the recurrent IFS. The construction in  \cite{HM} is based on the recurrent IFS and uses the technique of triangulation of the domain. Also, the authors deal with a special choice, namely affine maps, for  the functions $v_{i}: \mathbb{R}^{n} \times \mathbb{R}^{m} \to \mathbb{R}^{m}$ which take the role of $F_{i_1\ldots i_n}$ in our notation.  We  deal with more general functions (not necessarily affine) satisfying the required boundary conditions. More importantly,  the constructions  in \cite{PB1, HM} do not assist us for the $\alpha$-fractal function formalism and related approximation theoretic aspects of the multivariate FIFs that form the main focus of the current paper. 
	\end{remark}

	\section{A parameterized family of Multivariate fractal functions and associated fractal operator }
	Let $n \in \mathbb{N}$, $n \ge 2$ be fixed and $I_k=[a_k,b_k] \subset \mathbb{R}$ be a compact interval for $k=1,2,\dots,n$. Consider the $n$-dimensional  hyperrectangle  $\prod_{k=1}^{n}I_k$  and the space $\mathcal{C}\big(\prod_{k=1}^{n}I_k\big)$ endowed with the uniform norm. Let 
	a function $f\in \mathcal{C}\big(\prod_{k=1}^{n}I_k\big)$  be fixed, but arbitrary. We shall refer to this function as the \textit{seed function} or \textit{germ function}. 
	\subsection{Multivariate $\alpha$-fractal function}
	Here we obtain a parameterized family of fractal functions associated to a prescribed germ  function $f$ by using the idea of multivariate fractal interpolation enunciated in the previous section.  
	\par With a slight abuse of notation, consider the set  $$\Delta = \Big\{(x_{1,i_1},x_{2,i_2},\ldots,x_{n,i_n}) \in \prod_{k=1}^{n}I_k \subset \mathbb{R}^n:i_k \in \Sigma_{N_k,0},~k\in \Sigma_n\Big\},$$ where $a_k=x_{k,0}< x_{1,k} < \dots <x_{k,N_{k}}=b_k$ for each $k\in \Sigma_n:=\{1,2,\ldots,n\}$. Note that $\{x_{k,0},x_{1,k}, \dots,x_{k,N_{k}}\}$ forms a partition of the interval $[a_k,b_k]$ with the aid of which  $\Delta$ determines a partition of the hyperrectangle. Let us sample the germ function $f$ at the points in $\Delta$ and consider the data set 
	$$\Big\{\big(x_{1,i_1},\ldots,x_{n,i_n}, f(x_{1,i_1},\ldots,x_{n,i_n})\big) \in \prod_{k=1}^{n}I_k \times \mathbb{R}:i_k \in \Sigma_{N_k,0},~k\in \Sigma_n\Big\}.$$
	We shall denote the above data set also by $\Delta$. \\
	Suppose that the affine maps $u_{k,i_k}:I_{k} \rightarrow I_{k,i_k}$ is defined as follows
	\begin{equation}
		\begin{aligned}
			u_{k,i_k}(x) = a_{k,i_k}x + b_{k,i_k}, \quad i_k \in \Sigma_{N_k},~~ k \in \Sigma_{n},
		\end{aligned}
	\end{equation}
	where $a_{k,i_k}$ and $b_{k,i_k}$ are chosen such that the contractive maps $u_{k,i_k}$ satisfy (\ref{constraintu1}) and (\ref{constraintu4}).
	Choose a function $b \in \mathcal{C}\big(\prod_{k=1}^{n}I_k\big)$ such that for all $(j_1,j_2,\ldots, j_n) \in \prod_{k=1}^n\partial\Sigma_{N_k,0}$,
	\begin{equation}\label{base}
		\begin{aligned}
			b(x_{1,j_1},x_{2,j_2},\ldots,x_{n,j_n}) = f(x_{1,j_1},x_{2,j_2},\ldots,x_{n,j_n}).
		\end{aligned}
	\end{equation}
	Consider a continuous  map $\alpha:\prod_{k=1}^{n}I_k \to \mathbb{R}$ such that $\|\alpha\|_{\infty} < 1.$ As in the univariate counterpart, we call $b$ as the base function and $\alpha$ as the scaling function. Define
	\begin{equation}\label{Falpha}
		\begin{aligned}
			&F_{i_1 i_2 \ldots i_n}\big(x_1,x_2,\dots,x_n,y\big)\\=&~ f\big(u_{1,i_1}(x_1), u_{2,i_2}(x_2)\ldots u_{n,i_n}(x_n)\big)\\& + \alpha\big(u_{1,i_1}(x_1), u_{2,i_2}(x_2)\ldots u_{n,i_n}(x_n)\big)\big(y-b(x_1,x_2, \dots,x_n)\big).
		\end{aligned}
	\end{equation}
	For $(i_1,i_2,\dots,i_n) \in \prod_{k=1}^n \Sigma_{N_k}$ and $(j_1,j_2, \dots j_n) \in \prod_{k=1}^n \Sigma_{N_k,0}$, we have 
	
	\begin{equation*}
		\begin{aligned}
			&F_{i_1 i_2 \ldots i_n} \big(x_{1,j_1},\dots, x_{n,j_n}, f (x_{1,j_1},\dots, x_{n,j_n}) \big) \\= &~ f\big(u_{1,i_1}(x_{1,j_1}), u_{2,i_2}(x_{2,j_2})\ldots u_{n,i_n}(x_{n,j_n})\big)\\
			=& f\big(x_{1, \tau(i_1,j_1)}, \dots, x_{n, \tau(i_n,j_n)} \big),
		\end{aligned}
	\end{equation*}
	verifying that (\ref{constraintF}) holds. Using the condition $\|\alpha\|_\infty <1$, one can easily see that the contractivity condition prescribed in (\ref{constraintF1}) is satisfied. Let
	$ i_k \in \text{int}\Sigma_{N_k,0}$, $ 1\leq k\leq n$, $(i_1,i_2,\ldots, i_n) \in \prod_{k=1}^{n}\Sigma_{N_{k}}$ and $x_k^*=u_{k,i_k}^{-1}(x_{k,i_k})=u_{k,i_k+1}^{-1}(x_{k,i_k}).$ For any $y \in \mathbb{R}$, 
	
	\begin{equation*}
		\begin{aligned}	
			&~F_{i_1\ldots i_{k-1}i_ki_{k+1} \ldots i_n}(x_1,\ldots,x_{k-1}, x_k^*, x_{k+1},\ldots x_n,y) \\=&~	f\big(u_{1,i_1}(x_1), u_{2,i_2}(x_2)\ldots, u_{k-1,i_{k-1}}(x_{k-1}), u_{k,i_k}(x_k^*),u_{k+1,i_{k+1}}(x_{k+1}), \dots, u_{n,i_n}(x_n)\big)\\& + \alpha\big(u_{1,i_1}(x_1), u_{2,i_2}(x_2)\ldots, u_{k-1,i_{k-1}}(x_{k-1}), u_{k,i_k}(x_k^*),u_{k+1,i_{k+1}}(x_{k+1}), \dots, u_{n,i_n}(x_n)\big)\\&~ .\big(y-b(x_1,x_2, \dots,x_n)\big)\\
			=&~f\big(u_{1,i_1}(x_1), u_{2,i_2}(x_2)\ldots, u_{k-1,i_{k-1}}(x_{k-1}), u_{k,i_k+1}(x_k^*),u_{k+1,i_{k+1}}(x_{k+1}), \dots, u_{n,i_n}(x_n)\big)\\& + \alpha\big(u_{1,i_1}(x_1), u_{2,i_2}(x_2)\ldots, u_{k-1,i_{k-1}}(x_{k-1}), u_{k,i_k+1}(x_k^*),u_{k+1,i_{k+1}}(x_{k+1}), \dots, u_{n,i_n}(x_n)\big)\\&~ .\big(y-b(x_1,x_2, \dots,x_n)\big)\\
			=&~F_{i_1\ldots i_{k-1}i_k+1i_{k+1} \ldots i_n}(x_1,\ldots,x_{k-1}, x_k^*, x_{k+1},\ldots x_n,y).
		\end{aligned}
	\end{equation*}
	Therefore, the  functions $F_{i_1 i_2\dots i_n}$ in (\ref{Falpha}) satisfy  the conditions prescribed in (\ref{constraintF})-(\ref{constraintF1}) and the matching condition in (\ref{matching}). Ergo, by Theorem (\ref{FIF}) there exists a unique fractal interpolation function, which we shall denote by $f^{\alpha}_{\Delta,b}:\prod_{k =1}^n I_k \to \mathbb{R}$, such that it satisfies the self-referential functional equation
	\begin{equation}\label{selfref}
		\begin{aligned}
			f^{\alpha}_{\Delta,b}(x_1,x_2,\dots,x_n) = &~ f(x_1,x_2,\dots,x_n) +\\ &\alpha(x_1,x_2,\dots,x_n)\Bigg((f^{\alpha}_{\Delta,b}-b)\big(u_{1,i_1}^{-1}(x_1),u_{2,i_2}^{-1}(x_2),\ldots u_{n,i_n}^{-1}(x_n)\big)\Bigg),
		\end{aligned}
	\end{equation} 
	for all  $ (x_1,x_2,\dots,x_n) \in \prod_{k=1}^n I_{k,i_k}$ and $ (i_1,i_2,\ldots, i_n) \in \prod_{k=1}^n \Sigma_{N_k}$. \\
	As previously, with the notation 
	\begin{equation*}
		\begin{aligned}
			X= &~ (x_1,x_2,\dots,x_n),\\ 
			u_{i_1i_2\ldots i_n}^{-1}(X) =&~ \big(u_{1,i_1}^{-1}(x_1),u_{2,i_2}^{-1}(x_2),\ldots u_{n,i_n}^{-1}(x_n)\big),
		\end{aligned}
	\end{equation*}
	we may write the functional equation for the fractal function $f_{\Delta,b}^\alpha$ as follows.
	\begin{equation}\label{selfrefnewA}
		f^{\alpha}_{\Delta,b}(X)=f(X)+ \alpha(X) (f^{\alpha}_{\Delta,b}-b)\big(  u_{i_1i_2\ldots i_n}^{-1}(X)\big),
	\end{equation}
	for $X \in \prod_{k=1}^n I_{k,i_k}$ and $ (i_1,i_2,\ldots, i_n) \in \prod_{k=1}^n \Sigma_{N_k}$. It is worth to mention that
	\begin{equation*}
		f^{\alpha}_{\Delta,b}(x_{1,i_1},x_{2,i_2},\ldots,x_{n,i_n}) = f(x_{1,i_1},x_{2,i_2},\ldots,x_{n,i_n}), \quad \forall~~ i_k \in \Sigma_{N_k,0},~k \in \Sigma_n.
	\end{equation*}
	\begin{definition}
		The aforementioned continuous function $f^{\alpha}_{\Delta,b}:\prod_{k=1}^n I_k \to \mathbb{R}$ is referred to as the multivariate $\alpha$-fractal interpolation function corresponding to the seed function $f$,  associated with the scale function $\alpha,$ partition $\Delta$ and  base function $b.$ It can be viewed as a fractal perturbation of the germ function $f.$
	\end{definition}
	\begin{remark}
		With different admissible choices of parameters $\Delta$, $\alpha$ and $b$, in fact, we obtain a parameterized family of self-referential functions $\{f_{\Delta,b} ^\alpha\}$, each of which interpolates the germ function at points in $\Delta$. 
	\end{remark}Using the self-referential functional equation satisfied by $f_{\Delta,b}^\alpha$, it is straightforward to see the following inequality. 
	
	\begin{proposition}\label{error1}
		Let $f \in \mathcal{C}(\prod_{k=1}^{n}I_k)$ and the parameters $\alpha$, $\Delta$ and $b$ be fixed as in the construction above. Then,
		\begin{equation}
			\begin{aligned}
				\|f^{\alpha}_{\Delta,b} - f\|_{\infty} \leq \|\alpha\|_{\infty}\|f^{\alpha}_{\Delta,b}-b\|_{\infty}.
			\end{aligned}
		\end{equation}
	\end{proposition}
	Proposition \ref{error1} in conjunction with the triangle inequality yields the following upper bound for the uniform distance between a germ function $f$ and its fractal counterpart $f^{\alpha}_{\Delta,b}$. 
	\begin{proposition}\label{error2}
		Let $f \in \mathcal{C}(\prod_{k=1}^{n}I_k)$.  Assume that the parameters $\alpha$, $\Delta$ and $b$ are fixed as in the construction above. Then,
		\begin{equation}
			\begin{aligned}
				\|f^{\alpha}_{\Delta,b} - f\|_{\infty} \leq \frac{\|\alpha\|_{\infty}}{1-\|\alpha\|_{\infty}}\|f-b\|_{\infty}.
			\end{aligned}
		\end{equation}
	\end{proposition}
	\begin{remark}\label{convergenceremark}
		Bearing Proposition \ref{error2} in mind, we have the following results that point to  the approximation of the germ function $f\in \mathcal{C}(\prod_{k=1}^{n}I_k)$ with its fractal counterparts.
		\begin{enumerate}
			\item Let the base function $b$ and the partition $\Delta$ in the construction of $\alpha$-fractal interpolation functions corresponding to $f$ be fixed. Assume that $(\alpha_m)_{m\in\mathbb{N}}$ is a sequence of scale functions such that $\|\alpha^m\|_\infty<1$ for all $m\in \mathbb{N}$ and  $\|\alpha_m\|_{\infty} \to 0$ as $m \to \infty.$ Then, $\|f^{\alpha_m}_{\Delta,b} - f\|_{\infty} \to 0$ as $ m \to \infty.$
			\item  Let the scale function $\alpha$ and the partition $\Delta$ in the construction of $\alpha$-fractal interpolation functions corresponding to $f$ be fixed. Assume that 
			$(b_m)_{m\in\mathbb{N}}$ be sequence of base functions such that $\| f-b_m\|_{\infty} \to 0$ as $m \to \infty.$ Then $\|f^{\alpha}_{\Delta,b_m} - f\|_{\infty} \to 0$ as $m \to \infty$.
		\end{enumerate}
	\end{remark}
	\subsection{Fractal Operator}
	Here we choose the base function $b$ in the construction of $\alpha$-fractal  function $f_{\Delta,b}^\alpha$ through  an operator  $L:\mathcal{C}(\prod_{k=1}^{n}I_k) \to \mathcal{C}(\prod_{k=1}^{n}I_k)$ as follows. Assume that  $L:\mathcal{C}(\prod_{k=1}^{n}I_k) \to \mathcal{C}(\prod_{k=1}^{n}I_k)$ satisfies
	$$L(f)(x_{1,j_1},\ldots, x_{n,j_n}) = f(x_{1,j_1},\ldots, x_{n,j_n}), \quad \forall~~ (j_1,\ldots,j_n) \in \prod_{k=1}^n\partial\Sigma_{N_k,0}.$$
	Such an operator will be referred to as an admissible operator.  Take $b = L(f)$.
	In this case, we denote the $\alpha-$fractal function $f^{\alpha}_{\Delta,b}$ corresponding to $f$ as $f^{\alpha}_{\Delta,L}$. In contrast to the univariate case studied in the literature, here, in general,  the operator $L$ that defines the parameter map $b$ is not necessarily linear. This will help the fractal operator to have access to the realm of nonlinear operator theory. 
	\begin{definition}
		Fix a partition $\Delta$, scale function $\alpha$ and operator $L:\mathcal{C}(\prod_{k=1}^{n}I_k) \to \mathcal{C}(\prod_{k=1}^{n}I_k)$ as mentioned above. The operator $$\mathcal{F}^{\alpha}_{\Delta,L}:\mathcal{C}(\prod_{k=1}^{n}I_k) \to \mathcal{C}(\prod_{k=1}^{n}I_k),\quad \mathcal{F}^{\alpha}_{\Delta,L}(f)= f^{\alpha}_{\Delta,L},$$ is called the (multivariate) fractal operator.
	\end{definition}
	\begin{remark}
		By the construction of $\alpha$-fractal function given above, it follows that the fractal operator $\mathcal{F}^{\alpha}_{\Delta,L}$ interpolates $f$ at the points in the chosen partition $\Delta$. Further,	in view of Proposition \ref{error2}, we have 
		\begin{equation} \label{errA}
			\|\mathcal{F}^{\alpha}_{\Delta,L}(f) -f\|_{\infty} \leq \frac{\|\alpha\|_{\infty}}{1-\|\alpha\|_{\infty}}\|f-L(f)\|_{\infty}.
		\end{equation}	
		Thus, for a suitable choice of the scale function $\alpha$ and/or base function $b$,  $\mathcal{F}^{\alpha}_{\Delta,L}(f)$ approximates $f$ sufficiently well. Further, if $L(f)=f$ or $\alpha=0$, then $\mathcal{F}^\alpha_{\Delta, L} (f)=f.$ In particular, if $L=I$, the identity operator on $\mathcal{C}(\prod_{k=1}^{n}I_k)$, then $\mathcal{F}_{\Delta, L}^\alpha = I$. 
	\end{remark} 
	\begin{remark}\label{remnew0}
		As in the univariate setting \cite{M1}, it can be proved that if the operator  $L:\mathcal{C}(\prod_{k=1}^{n}I_k) \to \mathcal{C}(\prod_{k=1}^{n}I_k)$ is a linear operator, then the corresponding  fractal operator $\mathcal{F}^{\alpha}_{\Delta,L}:\mathcal{C}(\prod_{k=1}^{n}I_k) \to \mathcal{C}(\prod_{k=1}^{n}I_k)$ is also linear. Furthermore, if $L$ is a bounded linear operator, then it follows from (\ref{errA})  that the  corresponding fractal operator 
		$\mathcal{F}^{\alpha}_{\Delta,L}$ is also a bounded linear operator.  In the remaining part of this section, we  bring to light a few properties of the fractal operator beyond the familiar terrain of bounded linear operators. 
	\end{remark}
	Let us recall a pair of basic definitions from the perturbation theory of operators. For the duration of this section,  let $X$ and $Y$ be two normed linear spaces over the same field $\mathbb{K}= \mathbb{R}~ \text{or}~ \mathbb{C}$.
	
	\begin{definition}
		Let $A,B$ be two operators between $X$ and $Y$. Then $A$ is called relatively bounded with respect to $B$ (or simply $B$-bounded) if there exist nonnegative constants $a,b$ such that 
		\begin{equation}\label{relbdd}
			\|A(x)\| \leq a\|x\| +b\|B(x)\|, \quad \forall~ x\in X.
		\end{equation}
		The infimum of all such values of $b$ is called the $B$-bound of $A$.
	\end{definition}

	\begin{definition}
		Let  $A,B:X \to Y$ be two operators  between $X$ and $Y$. Then $A$ is called relatively Lipschitz with respect to $B$ (or simply $B$-Lipschitz) if 
		\begin{equation}\label{relLip}
			\|A(x) - A(y)\| \leq a\|x - y\| +b\|B(x) - B(y)\|, \quad \forall~ x,y\in X,
		\end{equation}
		for some nonnegative constants $a$ and $b$. The infimum of all such values of $b$ is called the $B$-Lipschitz constant of $A$.
	\end{definition}
	
	\begin{remark}
		Let $b_0$ be the infimum of all values of $b$ satisfying (\ref{relbdd}) or (\ref{relLip}). Then (\ref{relbdd}) or (\ref{relLip}) may not hold with $b=b_0$, because $a$ may tend to infinity as $b$ approaches $b_0.$
	\end{remark}

	\begin{proposition}
		The fractal operator $\mathcal{F}^{\alpha}_{\Delta,L}:\mathcal{C}(\prod_{k=1}^{n}I_k) \to \mathcal{C}(\prod_{k=1}^{n}I_k)$ is relatively bounded with respect to $L$  with $L$-bound less than or equal to  $\dfrac{\|\alpha\|_{\infty}}{1-\|\alpha\|_{\infty}}.$ 
	\end{proposition}
	\begin{proof}
		From (\ref{errA}) we have 
		\begin{equation*}
			\|f^{\alpha}_{\Delta,L}-f\|_{\infty} \leq \frac{\|\alpha\|_{\infty}}{1-\|\alpha\|_{\infty}}\big(\|f\|_{\infty}+\|L(f)\|_{\infty}\big).
		\end{equation*}
		Thus,
		\begin{equation*}
			\| \mathcal{F}^{\alpha}_{\Delta,L}(f)   \|_\infty= \|f^{\alpha}_{\Delta,L}\|_{\infty} \leq \frac{1}{1-\|\alpha\|_{\infty}}\|f\|_{\infty}+\frac{\|\alpha\|_{\infty}}{1-\|\alpha\|_{\infty}}\|L(f)\|_{\infty},
		\end{equation*}
		completing the proof.
	\end{proof}
	As a consequence of the previous proposition, we have the following result that states that the nonlinear  fractal operator  $\mathcal{F}^{\alpha}_{\Delta,L}$ shares some basic boundedness properties of the operator $L$.

	\begin{corollary}\label{cor3.3}
		Consider the map $L:\mathcal{C}(\prod_{k=1}^{n}I_k) \to \mathcal{C}(\prod_{k=1}^{n}I_k)$  and the corresponding (multivariate) fractal operator $\mathcal{F}_{\Delta, L}^\alpha$.
		\begin{enumerate}
			
			\item If $L$ is topologically bounded (that is, $L$ maps bounded sets into bounded sets), then  $\mathcal{F}_{\Delta, L}^\alpha$ is also topologically bounded. 
			
			\item If $L$ is norm-bounded, that is, 
			$$\rho(L):= \max \Big\{ \sup_{f \neq 0} \frac{\|L(f)\|_\infty}{\|f\|_\infty},  \|L(0)\|_\infty\Big\} < \infty,$$ then  $\mathcal{F}_{\Delta, L}^\alpha$ is also norm-bounded. 
			\item If $L$ is quasibounded, that is, 
			$$[L]_Q := \limsup_{\|f\|\to \infty} \frac{\|L(f)\|_\infty}{\|f\|_\infty}< \infty,$$
			then $\mathcal{F}_{\Delta, L}^\alpha$ is quasibounded as well. 
		\end{enumerate} 
	\end{corollary}
	\begin{proposition}\label{resLip}
		The fractal operator $\mathcal{F}^{\alpha}_{\Delta,L}:\mathcal{C}(\prod_{k=1}^{n}I_k) \to \mathcal{C}(\prod_{k=1}^{n}I_k)$ is relatively Lipschitz  with respect to $L$ and its  $L$-Lipschitz constant  is less than or equal to  $\dfrac{\|\alpha\|_{\infty}}{1-\|\alpha\|_{\infty}}.$ In particular, 	if $L:\mathcal{C}(\prod_{k=1}^{n}I_k) \to \mathcal{C}(\prod_{k=1}^{n}I_k)$ is a Lipschitz operator, then so is the fractal operator $\mathcal{F}^{\alpha}_{\Delta,L}$ with its Lipschitz constant $|\mathcal{F}^{\alpha}_{\Delta,L}| \leq \frac{1+\|\alpha\|_{\infty}|L|}{1-\|\alpha\|_{\infty}},$ where 
		$|L|$ is the Lipschitz constant of $L$ defined by $|L|= \sup_{f\neq g} \dfrac{\|L(f)-L(g)\|_\infty}{\|f-g\|_\infty}$. 
	\end{proposition}
	\begin{proof}
		Let $f$ and $g$ be in $\mathcal{C}(\prod_{k=1}^{n}I_k)$. Using the functional equations for the fractal functions  $f^{\alpha}_{\Delta,L}$ and $ g^{\alpha}_{\Delta,L}$ \big(see (\ref{selfref})\big), 
		and some routine computations we have
		\begin{equation*}
			\begin{aligned}
				\big\| \mathcal{F}^{\alpha}_{\Delta,L}(f) - \mathcal{F}^{\alpha}_{\Delta,L}(g) \big\|_\infty =&~ \|f^{\alpha}_{\Delta,L} - g^{\alpha}_{\Delta,L} \|_\infty \\
				\le &~ \frac{1}{1-\|\alpha\|_{\infty}}\|f-g\|_{\infty} +\frac{\|\alpha\|_{\infty}}{1-\|\alpha\|_{\infty}} \| L(f)- L(g)\|_\infty,
			\end{aligned}
		\end{equation*}
		proving that $\mathcal{F}^{\alpha}_{\Delta,L}$ is relatively Lipschitz with respect to $L$. In particular, if $L$ is Lipschitz with the Lipschitz constant $|L|$, the previous inequality yields
		
		$$	\| \mathcal{F}^{\alpha}_{\Delta,L}(f) - \mathcal{F}^{\alpha}_{\Delta,L}(g) \|_\infty \le \frac{1}{1-\|\alpha\|_{\infty}}\|f-g\|_{\infty} +\frac{\|\alpha\|_{\infty}}{1-\|\alpha\|_{\infty}} |L|\| f- g\|_\infty,$$ proving that $\mathcal{F}^{\alpha}_{\Delta,L}$ is a Lipschitz operator. 
	\end{proof}

	\begin{lemma}(\cite{RW})
		Let $X$ be a Banach space, $A: X \to X$ be a Lipschitz operator and $|Id-A|<1$, where $Id$ is the identity operator on $X$.  Then $A^{-1}: X \to X $  is Lipschitz and 
		$|A^{-1}| \le \dfrac{1}{1- |Id-A|}.$
	\end{lemma}
	\begin{proposition}
		Assume that $L: \mathcal{C}(\prod_{k=1}^{n}I_k) \to \mathcal{C}(\prod_{k=1}^{n}I_k)$ is a Lipschitz operator and the scaling function $\alpha$ is such that $\| \alpha\|_\infty < \big(  2+ |L| \big)^{-1}$. Then the fractal operator $\mathcal{F}^{\alpha}_{\Delta,L}:\mathcal{C}(\prod_{k=1}^{n}I_k) \to \mathcal{C}(\prod_{k=1}^{n}I_k)$ is a Lipschitz isomorphism (surjective bilipschitz map), that is, 
		$\mathcal{F}^{\alpha}_{\Delta,L}$ is a bijective Lipschitz operator and its inverse $(\mathcal{F}^{\alpha}_{\Delta,L})^{-1}:\mathcal{C}(\prod_{k=1}^{n}I_k) \to \mathcal{C}(\prod_{k=1}^{n}I_k)$ is also Lipschitz. 
	\end{proposition}
	\begin{proof}
		Let $f,g \in \mathcal{C}(I).$ Using the functional equations (see (\ref{selfref})) for the self-referential counterparts to $f$ and $g$,  we obtain
		\begin{equation*}
			\begin{aligned}
				\|(Id - \mathcal{F}^{\alpha}_{\Delta,L})(f) - (Id - \mathcal{F}^{\alpha}_{\Delta,L})(g)\|_{\infty} \leq \|\alpha\|_{\infty}\big(\|L(f) - L(g)\|_{\infty} + \|f^{\alpha}_{\Delta,L} - g^{\alpha}_{\Delta,L}\|_{\infty}\big).
			\end{aligned}
		\end{equation*}	
		Therefore, via  Proposition \ref{resLip}
		\begin{equation*}
			\begin{aligned}
				|Id - \mathcal{F}^{\alpha}_{\Delta,L}| \leq &~  \|\alpha\|_{\infty}\big(|L| +|\mathcal{F}^{\alpha}_{\Delta,L}|\big)\\
				\leq &~ \|\alpha\|_{\infty}\Big(|L| + \frac{1+\|\alpha\|_{\infty}|L|}{1-\|\alpha\|_{\infty}}\Big),\\
				=&~  \frac{\|\alpha\|_{\infty}\big(1+|L|\big)}{1-\|\alpha\|_{\infty}}.
			\end{aligned}
		\end{equation*}
		The assertion is now immediate by the previous lemma. 
	\end{proof}

	\bibliographystyle{amsplain}

\end{document}